\documentclass[a4paper,11pt]{article}
\newcommand{\Title}{Coupling of Brownian motions in Banach spaces}

\usepackage{amsmath,amsthm}
\usepackage{graphicx}
\usepackage{verbatim}
\usepackage[margin=2cm]{geometry}
\usepackage[normalem]{ulem}
\usepackage{amsfonts,amssymb,amsopn,amscd}

\usepackage{url}
\usepackage[british]{babel}
\usepackage{natbib}

\usepackage[colorlinks]{hyperref}
\usepackage[figure,table]{hypcap} 
\hypersetup{
   pdfauthor={Elisabetta Candellero and Wilfrid S. Kendall},
   pdftitle={\Title},
   pdfsubject={Coupling on Banach spaces},
   bookmarksnumbered,
   pdfstartview={FitH},
   citecolor={blue},
   linkcolor={red},
   urlcolor={black},
   pdfpagemode={UseOutlines}
}

\usepackage{WSKmaths}
\graphicspath{{image/}}

\newtheorem*{thm*}{\rm\textbf{Theorem}}

\newcommand{\N}{\mathbb{N}}
\newcommand{\A}{\mathbf{A}} 
\newcommand{\B}{\mathbf{B}} 

\newcommand{\cB}{\widetilde\B} 
\newcommand{\BAN}{\mathcal{W}}
\newcommand{\BANd}[2]{\left\langle#1,#2\right\rangle_{\BAN^*;\BAN}}
\newcommand{\BANn}[1]{\left\|#1\right\|_\BAN}
\newcommand{\BANnt}[1]{\|#1\|_\BAN} 
\newcommand{\CM}{{\mathcal{H}_\mu}}
\newcommand{\CMn}[1]{\left\|#1\right\|_\CM}
\newcommand{\CMnt}[1]{\|#1\|_\CM} 
\newcommand{\CMd}[2]{\left\langle#1,#2\right\rangle_{\CM}}
\newcommand{\COV}{\operatorname{Cov}}

\newcommand{\vertiii}[1]{{\left\vert\kern-0.25ex\left\vert\kern-0.25ex\left\vert #1 
    \right\vert\kern-0.25ex\right\vert\kern-0.25ex\right\vert}}

\numberwithin{equation}{section}

\begin{document}

\title{\Title}

\author{Elisabetta Candellero\footnote{Dept. Statistics, University of Warwick, Coventry CV4 7AL, UK; \url{E.Candellero@warwick.ac.uk}.} 
\and 
Wilfrid S.\ Kendall\footnote{Dept. Statistics, University of Warwick, Coventry CV4 7AL, UK; \url{W.S.Kendall@warwick.ac.uk}.}%
{~}\footnote{WSK supported by EPSRC Research Grant EP/K013939.
This is a theoretical research paper and, as such, no new data were created during this study.
}}

\maketitle

\begin{abstract}
Consider a separable Banach space \(\BAN\) supporting a non-trivial Gaussian measure \(\mu\).
The following is an immediate consequence of the theory of Gaussian measure on Banach spaces:
there exist (almost surely) successful couplings of
two \(\BAN\)-valued Brownian motions \(\B\) and \(\cB\) begun at starting points \(\B(0)\) and \(\cB(0)\)
if and only if 
the difference \(\B(0)-\cB(0)\) of
their initial 
positions
belongs to the Cameron-Martin space \(\CM\) of \(\BAN\) corresponding to
\(\mu\).
For more general starting points,
can there be
a ``coupling at time \(\infty\)'',
such that almost surely \(\BANnt{\B(t)-\cB(t)} \to 0\) as \(t\to\infty\)?
Such couplings exist
if there exists a Schauder basis of \(\BAN\) which is also a \(\CM\)-orthonormal basis of \(\CM\).
We
propose
(and discuss some partial answers to) 
the question,
to what extent can one
express the probabilistic Banach space property 
``Brownian coupling at time \(\infty\) is always possible'' 
purely in terms of Banach space geometry?
\end{abstract}




\section{Introduction}
When can there be an (almost surely)
\emph{successful coupling} of two Brownian motions \(\B\) and \(\cB\) defined on a separable Banach space \(\BAN\)?
(When can \(\B\) and \(\cB\) be made to coincide at and after some random time \(\tau<\infty\)?)
Is a weaker kind of success more widely available?
The purpose of this paper is to explore this weaker kind of success, and to raise an interesting open question.

Naturally the answer to the first question depends on the initial 
displacement of \(\B\) 
relative 
to \(\cB\).
Expressed more precisely, 
given a \(\BAN\)-valued Brownian motion \(\B=\{\B(t):t\geq 0\}\) started at \(0\), 
for which \(x\in \BAN\) is it 
possible to construct a second Brownian motion 
\(\cB=\{\cB(t):t\geq 0\}\) starting at \(x\)
and such that \(\B(\tau+\cdot)=\cB(\tau+\cdot)\) after some random time \(\tau\)?
(The general case
\(\cB(0)-\B(0)=x\) 
follows by translation invariance.)

We establish criteria for 
relative displacements \(x\) 
permitting
almost surely {successful (classical) coupling},
also almost sure \emph{coupling at time \(\infty\)} 
(almost surely \(\BANnt{\cB(t)-\B(t)}\to0\) as \(t\to\infty\) when \(\BANnt{\cdot}\) is the norm of \(\BAN\): 
see Section \ref{sec:coupling_schauder}).
In both cases the coupling can be chosen to be an \emph{immersion} or \emph{Markovian coupling}:
martingales in the natural filtration of \(\B\), respectively the natural filtration of \(\cB\),
remain martingales in the joint natural filtration of \(\B\) and \(\cB\) \citep{Kendall-2013a}.

The question of successful classical coupling 
is
resolved
by recalling the notion of 
the \emph{Cameron-Martin} space of a Gaussian measure \(\mu\) on \(\BAN\) \citep{CameronMartin-1944}.
Given a Gaussian measure \(\mu\), there is a 
standard construction of a Hilbert space \(\CM\) densely embedded in \(\BAN\),
such that translations of \(\mu\) by elements of \(\CM\) are 
exactly those that induce translated measures which are absolutely continuous with respect to \(\mu\).
This
theory, together with the well-known Aldous inequality \citep[Lemma 3.6]{Aldous-1983},
immediately 
yields
the following.
\begin{thm*}[See Theorem \ref{thm:xinH} below]\emph{
Let \(\BAN\) be a separable Banach space with norm \(\BANnt{\cdot}\).
Consider a \(\BAN\)-valued Brownian motion \(\B\) started at \(0\) and an element \(x\in \BAN\). 
Another Brownian motion \(\cB\) can be constructed to start at \(x\) and almost surely meet \(\B\) within finite time
if and only if the relative initial displacement \(x\) lies in \(\CM\), 
for \(\mu=\Law{\B(1)}\).
In that case the ``fastest possible'' coupling time is realized as the hitting time \(\tau\) 
of \(\CMd{x}{\B(\tau)}\) on \(\tfrac12\CMnt{x}^2\).
}\end{thm*}
If the initial displacement does not 
lie in
\(\CM\) then 
in many cases a weaker form of coupling
is still available,
namely ``coupling at time \(\infty\)''.
\begin{thm*}[See Theorem \ref{thm:schauder_WnotH} below]\emph{
Let \(\BAN\) be a separable Banach space with norm \(\BANnt{\cdot}\). 
Consider a \(\BAN\)-valued Brownian motion \(\B\) started at \(0\),
such that the associated \(\CM\) contains an orthonormal basis which is also a Schauder basis for \(\BAN\).
For \emph{any} \(x\in \BAN\), one can construct another Brownian motion \(\cB\) started at \(x\) 
which almost surely couples with \(\B\) at time \(\infty\):
\begin{equation}\label{eq:coupling-at-time-infinity}
 \Prob{\BANn{\cB(t) -\B(t)} 
 \;\to\;
 0 \text{ as } t\to\infty}\quad=\quad1\,.
\end{equation}
}\end{thm*}
\noindent
The Schauder basis condition is not necessary for coupling at time \(\infty\): 
we will discuss various extensions
and conclude by raising the open 
Question \ref{qn:open}: 
whether in fact this Theorem might hold for general Banach spaces \(\BAN\)
regardless of 
whether 
the associated \(\CM\) contains a Schauder basis for \(\BAN\); and if not,
then what 
Banach space 
geometries
imply existence of Brownian coupling at time \(\infty\)?

\begin{rem}
When \eqref{eq:coupling-at-time-infinity} holds, in whatever Banach space context,
we say that the Brownian motions \(\B\) and \(\cB\) 
(begun at \(\B(0)\) and \(\widetilde\B(0)\))
\emph{couple at time \(\infty\)} almost surely.
\end{rem}
The paper is organized as follows. 
Section \ref{sec:gaussian_measures} surveys the basic theory of Gaussian measures and Brownian motions on Banach spaces.
Section \ref{sec:coupling_finite} treats the case when the initial displacement \(x\) to be in \(\CM\), 
while the case of \(x\in \BAN\setminus \CM\) is discussed in 
Section \ref{sec:coupling_schauder} 
(under the assumption that \(\BAN\) admits a Schauder basis which lies in \(\CM\) and which also forms an orthonormal basis for \(\CM\)).
Section \ref{sec:conc} discusses extensions and future work and raises the Open Question \ref{qn:open}.

\section{Gaussian measures and Banach-valued Brownian motion}\label{sec:gaussian_measures}
Recall the following
facts about Gaussian measures in Banach space.
Proofs can be found
for example in
\cite{Kuo-1975}, \cite{Hairer-2009}, \cite{Stroock-2011}, \cite{Eldredge-2012}.
A \emph{Gaussian probability measure} \(\mu\) on a separable Banach space \(\BAN\) is a Borel measure on \(\BAN\) such that
the push-forward 
\(\ell\#\mu\)
by any continuous linear functional \(\ell\in\BAN^*\) is a 
Gaussian probability measure on \(\Reals\).
If 
\(\ell\#\mu\)
has mean \(0\) for every \(\ell\) then \(\mu\) is \emph{centered}.
For simplicity, consider only  {centered} Gaussian measures
which are \emph{non-degenerate}: \(\ell\#\mu\) is non-degenerate for all non-zero \(\ell\). 

The \emph{covariance operator} of \(\mu\) is a map \(\COV_{\CM} : \BAN^* \times \BAN^*\to \Reals\) defined by
\[
\COV_{\CM} (\ell_1, \ell_2) \quad=\quad \int_\BAN\ell_1(x)\ell_2(x)\;\mu(\d{x})\,.
\]
This can be identified with a bounded linear operator \(\widehat{\COV}_{\CM} : \BAN^* \to \BAN^{* *}\) as follows:
\begin{equation}\label{eq:def_widehat_C_mu}
\left(\widehat{\COV}_{\CM} \ell_1\right)(\ell_2) \quad=\quad
\COV_{\CM} (\ell_1,\ell_2)\,.
\end{equation}
In particular Fernique's Theorem (\citealp[Section III.3]{Kuo-1975}) holds: 
the range of \(\widehat{\COV}_{\CM}\) lies in \(\BAN\hookrightarrow\BAN^{**}\) (where ``$\hookrightarrow$'' denotes a continuous embedding)
and \(\widehat{\COV}_{\CM} \ell\) 
is
a barycentre:
\begin{equation}\label{eq:int_x_l_mu}
\widehat{\COV}_{\CM} \ell \quad=\quad \int_{\BAN} \ell(x)\;x\;\mu(\d{x})
\quad=\quad
\int_\BAN \BANd{\ell}{x} \; x \; \mu(\d{x})\,.
\end{equation}
Here \(\BANd{\ell}{x}=\ell(x)\) expresses the duality between \(\BAN^*\) and \(\BAN\). 

Fernique's theorem actually shows that $\widehat{\COV}_{\CM}$ is bounded as linear operator $ \BAN^\ast \to \BAN$ \citep[Section 3.1]{Hairer-2009}.

\subsection{The Cameron-Martin space}
The (non-degenerate) Gaussian probability measure \(\mu\) is canonically associated 
with its \emph{Cameron-Martin} space \(\CM\); a Hilbert space densely and continuously embedded in the Banach space \(\BAN\),
so \(\CM\hookrightarrow\BAN\)
(\citealp[Section 3.2]{Hairer-2009}; \citealp[Section 4.3]{Eldredge-2012}). 
The triple \((\BAN , \CM , \mu)\) is an \emph{abstract Wiener space}.
More precisely, let 
\begin{equation}\label{eq:def_H_ring}
\mathring{\CM} \;=\;
\{ h\in \BAN \ : \ \textnormal{there exists }h^*\in \BAN^* \textnormal{ such that }\COV_{\CM}(h^* ,\ell)=\ell(h), \ \textnormal{for all }\ell \in \BAN^*\}\,.
\end{equation}
Then \(\mathring{\CM}\) is a pre-Hilbert space when endowed with the inner product defined by
\begin{equation}\label{eq:def_scal_prod}
\CMd{h}{k} \quad=\quad
\COV_{\CM}(h^* ,k^*)\,.
\end{equation}
Its Hilbert-space completion under the inner product \eqref{eq:def_scal_prod} 
is the \emph{Cameron-Martin space} of \(\BAN\) furnished with \(\mu\).
As a Hilbert space, \(\CM\) has the  norm
\begin{equation}\label{eq:norm_H}
\CMn{h}^2 \quad=\quad
\CMd{h}{h} \quad=\quad
\COV_{\CM}(h^* ,h^*)\,.
\end{equation}
Fernique's Theorem and completeness of \(\BAN\) imply that the embedding \(A:\CM \hookrightarrow \BAN\)
is continuous and injective
\citep[Section 3]{Hairer-2009}. Note that \(A^*A\) is the extension of \(\COV_\CM\) to \(\CM\times\CM\)
\begin{rem}\label{rem:dual_into_CM_space}
$\mathring{\CM}\) coincides with the range of \(\widehat{\COV}_\mu\) defined in \eqref{eq:def_widehat_C_mu}. 
Furthermore, in view of \eqref{eq:def_widehat_C_mu} and  \eqref{eq:def_H_ring}, 
we have \(\BAN^* \hookrightarrow \CM \hookrightarrow \BAN\) using continuous embeddings of dense image.
\end{rem}
In particular, for any \(x\in \mathring\CM\) there is an element \(x^*\in \BAN^*\) associated to \(x\) for which 
\begin{equation}\label{eq:c_mu}
\COV_{\CM}(x^*,\ell)\quad=\quad
\ell(x) \quad=\quad \BANd{\ell}{x}, \quad \text{ for all } \ell\in \BAN^*.
\end{equation}
The following result 
is
key 
for analyzing the possibility of successful Brownian coupling.%
\begin{thm}[Feldman-Hajek Theorem, see {\citealp[Theorem 3.1]{Kuo-1975}}]\label{thm:cameron-martin}
For \(w\in \BAN\) define the translation map \(T_w: \BAN \to \BAN\) by \(T_w(x)=x+w\). 
If \(w\in \CM\) then the push-forward measure 
\({T_w}\#\mu\)
is absolutely continuous with respect to \(\mu\),
while if \(w\not\in \CM\) then the push-forward measure and \(\mu\) are mutually singular.
Moreover the join of the probability distributions \(\mu\)
and \({T_w}\#\mu\) has total mass given by
twice the probability that a standard normal random variable exceeds the value \(\tfrac12\CMn{w}\).
\end{thm}

\subsection{Brownian motion on a Banach space}
\citet[Chapter 8]{Stroock-2011} 
uses these
considerations to 
define
Brownian motion (a process with stationary Gaussian increments independent of the past, and continuous sample paths).
Let \(K(\CM)\) be the Hilbert space of absolutely continuous functions \(h:[0,\infty ) \to \CM\) 
with almost-everywhere defined derivative \(\dot{h}\),
such that \(h(0)=0\) and \(\|h\|_{K(\CM)}^2=\|\dot{h}\|_{L^2([0,\infty ),\CM)}^2=\int_0^\infty\CMnt{\dot{h}}^2\d{t}<\infty\).
Define
\[
\mathcal{C}_\BAN\quad=\quad
\left \{\theta:[0,\infty ) \to \BAN \textnormal{, such that } 
\theta\textnormal{ is continuous and } \lim_{t\to \infty}\frac{\BANn{\theta(t)}}{t}=0\right \}
\,,
\]
which becomes a Banach space when endowed with the norm 
\[
\|\theta\|_{\mathcal{C}_\BAN}\quad=\quad
\sup_{t\geq 0}\frac{1}{1+t}{\BANn{\theta}}\,.
\]
Each centered non-degenerate Gaussian measure then corresponds to a Brownian motion:
\begin{thm}[\citealp{Stroock-2011}, Theorem 8.6.1] 
Given \(K(\CM)\) and \(\mathcal{C}_\BAN\) as above, 
there is a {unique} measure \(\mu_{\BAN}\) such that \((\mathcal{C}_\BAN,K(\CM),\mu_{\BAN})\)
is an abstract Wiener space.
\end{thm}
\begin{thm}[\citealp{Stroock-2011}, Theorem 8.6.6] 
Choose \(\Omega=\mathcal{C}_\BAN\). 
Let \(\mathcal{F}\) be the Borel \(\sigma\)-algebra of \(\BAN\), 
and let \(\mathcal{F}_t\) be the \(\sigma\)-algebra generated by 
\(\{\theta(s) \, : \,s\in [0,t]\}\). 
Then the triple \(((\theta(t):t\geq0),\mathcal{F}_t,\mu_{\BAN})\) is a \(\BAN\)-Brownian motion.
Conversely, if \(((B(t):t\geq0),\mathcal{F}_t,\mathbb{P})\) is any \(\BAN\) Brownian motion, 
then \(\Prob{B(\cdot) \in \mathcal{C}_\BAN}=1\)
and \(\mu_{\BAN}\) is the \(\mathbb{P}\) distribution of \(\omega \mapsto B(\cdot ,\omega)\).
\end{thm}

\subsection{Decomposition of Brownian motions}\label{sec:decomposition_BMs}
\citet{Stroock-2008} 
connects
\(\CM\)-orthogonal decompositions 
to Brownian independence.
\begin{thm}\label{thm:decomposition}
Suppose the Cameron-Martin space \(\CM\) 
admits an orthogonal decomposition into finite-dimensional subspaces 
\(\CM=\CM^{(1)}\oplus\CM^{(2)}\oplus\ldots\).
The corresponding \(\BAN\)-valued Brownian motion \(\B\) decomposes as a tuple of independent Brownian motions
\(\B=(\B_1,\B_2, \ldots)\) with \(\B_\ell\) living on the \(\BAN\)-closure of \(\CM^{(\ell)}\);
the resulting sum converges not only in \(L^2(\BAN,\Reals)\) but also almost surely in \(\BAN\).
\end{thm}

The proof of Theorem \ref{sec:decomposition_BMs} uses a fundamental result: 
\begin{thm}[Banach version of Marcinkiewicz' Theorem]\label{thm:Marcinkiewicz}
Let \(X\) be an \(\BAN\)-valued random variable, measurable with respect to a \(\sigma\)-algebra \(\mathcal{F}\).
Suppose that \(X\in L^p(\BAN, \Reals)\) for some \(p\in [1,\infty)\).
For any non-decreasing sequence \(\{\mathcal{F}_n:n=1,2,\ldots\}\) of sub-$\sigma\)-algebras of \(\mathcal{F}$,
\[
\lim_{n\to \infty}\mathbb{E}_\mu\left[X\mid \mathcal{F}_n\right] \quad=\quad 
\mathbb{E}_\mu\left[X \Big| \bigcup_{n=1}^\infty \mathcal{F}_n\right] \quad \text{both \(\mu\)--almost surely and also in }L^p(\BAN,\Reals)\,.
\]
Furthermore, if \(X\) is \(\bigcup_{n=1}^\infty \mathcal{F}_n\)-measurable then 
\[
\lim_{n\to \infty }\mathbb{E}_\mu\left[X\mid \mathcal{F}_n\right] \quad=\quad X \quad \text{both \(\mu\)--almost surely and also in }L^p(\BAN,\Reals)\,.
\]
\end{thm}
A proof of Theorem \ref{thm:Marcinkiewicz} is given in \citet[Theorems 1.14 and 1.30]{Pisier-2016};
a proof of the original (non-Banach) Marcinkiewicz' Theorem is given in \citet[Chapter 5]{Stroock-2011}.

We spell out the proof of Theorem \ref{thm:decomposition} to help establish notation.
\begin{proof}[Proof of Theorem \ref{thm:decomposition}]\text{}


\noindent\emph{Step 1.} 
\emph{Separability means that \(\BAN^*\) generates the Borel \(\sigma\)-algebra of \(\BAN\).}
The elements of \(\BAN^*\) are continuous functions on \(\BAN\), hence are measurable with respect to the Borel \(\sigma\)-algebra \(\Borel_\BAN\) of \(\BAN\).
Being separable, 
\(\BAN\)
is generated by a countable system of open balls; 
by separability, each of these balls can be expressed as the intersection of countably many half-spaces of \(\BAN\). 
Hence \(\Borel_\BAN\) is the \emph{smallest} \(\sigma\)-algebra 
making
all maps \(\lambda : \theta\mapsto \BANd{\theta}{\lambda}=\lambda(\theta)\) measurable
for all \(\lambda\in\BAN^*\).

\medskip

\noindent\emph{Step 2.} 
\emph{There is a natural isometry embedding \(\CM\) in \(L^2(\BAN,\mu)\).}
From \eqref{eq:def_H_ring}, for every \(\lambda\in \BAN^*\) there is a \(h_\lambda\in \CM\) such that 
\(\lambda(h)=\BANd{\lambda}{h}=\CMd{h}{h_\lambda}\).
Therefore, there is a natural isometric injection
\[
\mathcal{I}:\BAN^*\quad\to\quad L^2(\BAN,\mu)
\]
given by \(\left(\mathcal{I}h_\lambda\right)(\theta)=\langle \theta,\lambda\rangle_{\BAN,\BAN^\ast} \), defined for any \(\lambda \in \BAN^*\), 
and this 
extends
to \(\mathcal{I}:\CM\to L^2(\BAN,\mu)\).
Denoting by \(\iota\) the dense embedding \(\BAN^* \to \CM\),
\(\CMn{\iota(\lambda)}= \|\lambda\|_{L^2(\BAN, \mu)}\).
The extension \(\mathcal{I}:\CM\to L^2(\BAN,\mu)\) is 
the \emph{Paley--Wiener map}.

\medskip

\noindent\emph{Step 3.}
\emph{The union of the Paley-Wiener maps of the summands is dense in the Paley-Wiener image of \(\CM\).}
Given 
the orthogonal decomposition
\begin{equation}\label{eq:hilbert_decomposition}
\CM\quad=\quad
\bigoplus_{\ell=1}^\infty \CM^{(\ell)}
 \quad=\quad
\CM^{(1)}\oplus\CM^{(2)}\oplus\ldots\,,
\end{equation}
let 
\(h_1^{(\ell)}\), \(h_2^{(\ell)}\), \ldots, \(h_{k_\ell}^{(\ell)}\) be an orthonormal basis of \(\CM^{(\ell)}\).
By Steps 1 and 2, 
since the Paley-Wiener map is an isometry and each \(\CM^{(\ell)}\) is finite dimensional,
the \(\mathcal{I}h_1^{(\ell)}, \mathcal{I}h_2^{(\ell)}, \ldots, \mathcal{I}h_{k_\ell}^{(\ell)}\)
form an orthonormal basis for
the image \(\mathcal{I}\CM^{(\ell)}\). 
By \eqref{eq:hilbert_decomposition} the following linear span is \(L^2(\BAN,\mu)\)--dense in \(\mathcal{I}\CM\):
\[
\bigcup_{\ell=1}^\infty \left\{
\mathcal{I}h_1^{(\ell)}, \mathcal{I}h_2^{(\ell)}, \ldots, \mathcal{I}h_{k_\ell}^{(\ell)}
\right\}
\,.
\]

\medskip

\noindent\emph{Step 4.} 
\emph{We use the orthogonal decomposition \eqref{eq:hilbert_decomposition} 
to establish a filtration of \(\sigma\)-algebras over \(\BAN\):}
\(\mathcal{F}_1 \leq \mathcal{F}_2\leq \ldots\) is a filtration of \(\sigma\)-algebras 
on \(\BAN\) given by
\[
\mathcal{F}_n\quad=\quad
\sigma\left\{
(\mathcal{I}h_1^{(1)},\mathcal{I}h_2^{(1)},\ldots, \mathcal{I}h_{k_1}^{(1)}) ,
\ldots,
(\mathcal{I}h_1^{(n)},\mathcal{I}h_2^{(n)},\ldots, \mathcal{I}h_{k_n}^{(n)}),
\ldots
\right\}\,.
\]
Set \(\mathcal{F}=\sigma \left\{\mathcal{F}_1,\mathcal{F}_2\ldots\right\}\).
By \emph{Step 3}, all the $\mathcal{I}h_\lambda$ functions 
are measurable with respect to the $\mu$-completion of 
$\mathcal{F}$.
Moreover, by \emph{Step 2}, $\mathcal{I}h_\lambda $ maps $\theta$ to \(\langle \theta,\lambda\rangle_{\BAN,\BAN^\ast}\), 
thus $\mathcal{I}h_\lambda $ is also measurable with respect to \(\Borel_\BAN\).
Furthermore, \emph{Step 1} implies that \(\Borel_\BAN\) is the \emph{smallest} $\sigma$-algebra 
with respect to which all such maps are measurable.
This implies that \(\Borel_\BAN\) is contained in the \(\mu\)-completion of \(\mathcal{F}\).

\medskip

\noindent\emph{Step 5.} 
\emph{Compute the \(\mathcal{F}_n\)-conditional expectation
of a \(\mu\)-random choice of \(\theta\in\BAN\).}
Set
\[
S_n(\theta) \quad=\quad 
\sum_{\ell=1}^n \left ( \sum_{j}  h_j^{(\ell)} \times \left(\mathcal{I}h_j^{(\ell)}\right)(\theta)\right )\,.
\]
On each summand
\(\mu\) induces a Gaussian measure,
so
compute summand by summand to obtain 
\(S_n(\theta) = \mathbb{E}_\mu\left[\theta \mid \mathcal{F}_n\right]\).

\medskip

The proof of the theorem is completed by combining Theorem \ref{thm:Marcinkiewicz} (the Banach version of Marcinkiewicz' Theorem)
with Steps 1--5: 
both $\mu$--almost surely and in $L^2(\BAN, \mu)$,
\[
\lim_n S_n(\theta)\quad =\quad 
\lim_n \mathbb{E}_\mu\left[\theta \mid \mathcal{F}_n\right]
\quad=\quad \mathbb{E}_\mu\left[\theta \mid \mathcal{F}\right]
\quad=\quad \theta\,. 
\]
\end{proof}

\begin{rem}
 The summands \(\CM^{(\ell)}\) 
can also be
infinite-dimensional,
since
infinite-dimensional summands can 
be
expressed
as
orthogonal direct sums of finite-dimensional subspaces.
\end{rem}

\section{Coupling of Brownian motion within finite time}\label{sec:coupling_finite}
Recall that \(\B(0)=0\) and \(\cB(0)=x\) and denote by \(\dist_{\operatorname{TV}}(\nu,\nu')\) the \emph{total variation distance} between the probability measures \(\nu\) and \(\nu'\). 
The \citeauthor{Aldous-1983} inequality implies
\begin{equation}\label{eq:aldous_ineq}
\Prob{\cB \textnormal{ meets }\B \textnormal{ by time }t}\quad\leq\quad 
1-\dist_{\operatorname{TV}}(\cB(t),\B(t)).
\end{equation}
%
The Feldman-Hajek Theorem (Theorem \ref{thm:cameron-martin} above),
thus enforces a fundamental restriction on the possibility of coupling Brownian motions in Banach space:
if \(x\) does not lie in \(\CM\) then the distributions of \(\B(t)\) and \(\cB(t)\) are mutually singular,
and therefore \(\mathbb{P}[~{\cB \textnormal{ meets }\B \textnormal{ by time }t}~]=0\), for all \(t\) and all possible couplings.
On the other hand, 
if \(x\) does lie in \(\CM\) then there is a \emph{maximal} successful coupling:
maximal in the sense that \eqref{eq:aldous_ineq} becomes an equality for all \(t\).

\subsection{Cameron-Martin Reflections in Banach spaces}\label{sec:Cameron-Martin-Reflection}
Given \(x\in\CM\),
construct the \(\CM\)-projection \(h \mapsto \CMd{x}{h} x/\CMn{x}^2\).
Define the resulting \emph{Cameron-Martin reflection} \(R_x(y):\CM\to\CM\) by
\begin{equation}\label{eq:def_reflection}
R_x(y)\quad=\quad 
y- 2\frac{\CMd{x}{y}}{\CMn{x}^2}x  
\qquad \text{ for } y\in\CM\,.
\end{equation}
Note that \(R_x(x)=-x\).
By
Theorem \ref{thm:decomposition}, 
\(R_x\) 
produces
a \emph{reflected} Brownian motion \(\cB=x+R_x(\B)\):
we avoid having to consider the extension of \(R_x\) to all of \(\BAN\) by writing \(\CM=\text{Ker}(\mathbb{I}-R_x)\oplus\text{Ker}(\mathbb{I}+R_x)\), 
then applying Theorem \ref{thm:decomposition} to decompose \(\B=(\B_1,\B_2)\) accordingly
(so \(\B_1\) is a \(\text{Ker}(\mathbb{I}-R_x)\)-valued Brownian motion and \(\B_2\) is a \(\text{Ker}(\mathbb{I}+R_x)\)-valued Brownian motion 
-- in fact \(\B_1\) is essentially a one-dimensional Brownian motion \(\CMd{x}{B}x/\CMn{x}^2\)).
Finally 
set \(\cB=x+(\B_1,-\B_2)=x+\B-2\CMd{x}{B}x/\CMn{x}^2\).

\begin{rem}
\cite{Lindvall-1982a} introduced 
\emph{coupling by reflection} for real
Brownian motion. 
\cite{LindvallRogers-1986} 
adapted 
it to 
couple 
finite-dimensional diffusions.
Further generalizations
include
reflection coupling on Riemannian manifolds
and beyond
\citep{Kendall-1986b,Kendall-1986a,Kendall-1998e,Cranston-1991,VonRenesse-2004}.
\end{rem}

\subsection[Coupling in finite time holds exactly if initial displacement is in Cameron-Martin space]%
{Coupling in finite time holds exactly if initial displacement is in $\CM$}\label{sec:xinH}
We now establish the first and simplest coupling result for Brownian motion in Banach space.
\begin{thm}\label{thm:xinH}
A \(\BAN\)-valued Brownian motion \(\B\) with \(\B(0)=0\) can be coupled with 
another Brownian motion \(\cB\) with \(\cB(0)=x\neq0\) 
if and only if \(x\in \CM\),
where \(\mu=\Law{B(1)}\) is the Gaussian distribution of the Banach-space-valued random variable \(\B(1)\). 
Moreover, coupling can then succeed almost surely and there is even a maximal coupling time
(produced by a Cameron-Martin reflection based on \(x\))
which has the distribution of the 
first time \(T\) at which
\(\CMd{x}{\B(T)}= \tfrac12\CMn{x}^2\).
\end{thm}
\begin{proof}
One direction follows from the \citeauthor{Aldous-1983} inequality \eqref{eq:aldous_ineq} and Theorem \ref{thm:cameron-martin}
as noted at the start of Section \ref{sec:coupling_finite} above: if \(x\notin \CM\), then 
\[
\Prob{\B \textnormal{ meets }\cB\textnormal{ within finite time}}\quad=\quad0\,.
\]

On the other hand, consider \(x\in \CM\).
Given the Brownian motion \(\B\), 
construct \(\cB\) from \(\B\) using the reflection \(R_x\) as described in Section \ref{sec:Cameron-Martin-Reflection}.
Thus, until \(\cB\) and \(\B\) coincide,
let
\begin{equation}\label{eq:coupling_H}
\cB(t)\quad=\quad R_x(\B(t))+x
\quad=\quad
x + \B(t)-2\frac{\CMd{x}{\B(t)}}{\CMn{x}^2}x\,.
\end{equation}
Once \(\cB=\B\), the two Brownian motions 
evolve as a single process.
By Equation \eqref{eq:coupling_H}
\[
\A(t)\quad=\quad\cB(t)-\B(t)
\quad=\quad
\begin{cases}
x-2\frac{\CMd{x}{\B(t)}}{\CMn{x}^2}x & \text{ while }\cB(t)\neq\B(t)\,,\\
0 & \text{ once }\cB(t)=\B(t)\,,
\end{cases}
\]
is an element of the Cameron-Martin space \(\CM\), since it is a scalar multiple of \(x\in \CM\).
Hence,
\begin{multline*}
\Prob{\cB \textnormal{ meets }\B\textnormal{ by time }T}
\quad=\quad
\\
\Prob{\A(s)=0 \textnormal{ for some }s\leq T}
\quad=\quad
\Prob{\CMn{\A(s)} =0, \textnormal{ for some }s\leq T}\,.
\end{multline*}
Now \(\A\) is a difference of a Banach-valued Brownian motion and its reflection, stopped once they agree.
Therefore its absolute value 
is
a scalar Brownian motion of rate
\(4\), 
begun at \(\CMn{x}\),
stopped on reaching \(0\):%
\[
\CMn{\A(t)} 
\quad=\quad
\begin{cases}
\left(1-2\frac{\CMd{x}{\B(t)}}{\CMn{x}^2}\right)\CMn{x} & \text{ while }\cB(t)\neq\B(t)\,,\\
 0 & \text{ once } \cB(t)=\B(t)\,.
\end{cases}
\]
By the Reflection Principle,
the probability of coupling by time \(t\) is 
exactly the total mass of the join of probability distributions as given in 
Theorem \ref{thm:cameron-martin}.
Hence
this 
is indeed a maximal coupling.
\end{proof}

\begin{rem}
As
\(\CM\) is dense in \(\BAN\), we can produce \(\eps\)-approximate couplings of Brownian motions 
 begun at any two different starting points in \(\BAN\)
 (simply \(\BAN\)-approximate \(x\in\BAN\) by \(h\in\CM\)).
 Such \(\eps\)-approximate couplings are reminiscent of the Wasserstein variants of CFTP algorithms 
 introduced by \citet{Gibbs-2004} for image restoration,
 and 
 may be of use in future applications.
\end{rem}

\section[Coupling at time infinity and Schauder basis properties]%
{Coupling at time $\infty$ and Schauder basis properties}\label{sec:coupling_schauder}
The work of Section \ref{sec:coupling_finite} 
suggests the following natural question.
When will there exist a coupling between \(\B\) and \(\cB\)
that allows the two Brownian motions almost surely to ``\emph{couple at time \(\infty\)}'', whatever the starting 
points \(\B(0),\cB(0)\in\BAN\)?
In other words, when can
\[
\Prob{\lim_{t\to \infty} \BANn{\cB(t) -\B(t)} =0}\quad=\quad1\,.
\]
This question 
links
to the concept of \emph{Schauder bases} for Banach spaces.

\subsection{Schauder basis}
A Schauder basis for \(\BAN\) is a sequence \((e_k:k=1,2,\ldots)\) drawn from a Banach space \(\BAN\), 
such that each \(x\in \BAN\) admits a \emph{unique} decomposition as
the conditionally convergent sum
\(x=\sum_{k=1}^\infty \alpha_k e_k\)
for some coefficients \(\alpha_k\in \Reals\) depending on \(x\).
In fact \citet[pages 110-112]{Banach-1932} observed that \(\alpha_k\) depends continuously (and linearly) on \(x\) (see also \citealp[page 878]{McArthur-1972}).
Thus \(\alpha_k=\BANd{e_k^*}{x}\), where each \(e_k^*\in\BAN^*\) depends on the entire Schauder basis, with
the following holding as a conditionally \(\BAN\)-convergent sum
\begin{equation}\label{eq:biorthogonal}
 x\quad=\quad
\sum_{k=1}^\infty \BANd{e_k^*}{x}\; e_k\,.
\end{equation}
\begin{rem}
The convergence in Equation \eqref{eq:biorthogonal} is \emph{conditional} 
(it depends on the order of the sequence of basis vectors \(e_n\)) 
and must be interpreted using the \(\BAN\)-norm topology:
\[
\lim_{K\to \infty}\BANn{ w-\sum_{k=0}^K \BANd{e_k^*}{x}\; e_k}=0\,.
\]
\end{rem}
A Banach space with a Schauder basis is necessarily separable, but the converse is not true 
(see the celebrated counterexample of \citealp{Enflo-1973}).
Separable Banach spaces can admit other types of bases such as the \emph{Markushevich basis} 
(typically referred to as \emph{M-basis}), the \emph{Auerbach basis}, or simply a \emph{finite-dimensional decomposition}
(\citealp[Chapter 1]{HajekSantaluciaVanderwerffZizler-2008}, \citealp{Casazza-2001}). 
Markushevich bases will make an appearance in the concluding section \ref{sec:conc}, but we will not discuss the other notions here.

\subsection[Brownian motions coupled at time infty]{Brownian motions coupled at time $\infty$}
In the following, \(\B\) and \(\cB\)
denote coupled $\BAN$-valued Brownian motions  started at \(0\) and \(x\in \BAN\) respectively.
If \(x\in \BAN\setminus \CM\) then reflection coupling is not 
well defined!

\begin{thm}\label{thm:schauder_WnotH}
Let \(\BAN\) be a separable Banach space possessing a Schauder basis \((e_k: k=1,2,\ldots)\)
which also forms an orthonormal basis for \(\CM\). 
Then it is possible to construct a coupling at time \(\infty\) of Brownian motions \(\B\) and \(\cB\) started at any two given starting points,
so
\[
\Prob{\BANn{\B(t) -\cB(t)} \quad\to\quad 0\textnormal{ as }t\to\infty}\quad=\quad1\,.
\]
\end{thm}
 
\begin{rem}
Theorem \ref{thm:schauder_WnotH} holds for all possible \(x\in\BAN\), though
construction details 
depend on 
the Schauder expansion \eqref{eq:biorthogonal} of \(x\).
If \(x\not\in\CM\) then the coupling is 
actually maximal, albeit only in a 
degenerate sense, since
the distributions of \(\B(t)\) and \(\cB(t)\) in \eqref{eq:aldous_ineq} 
are
then mutually singular for all time \(t>0\).
\end{rem}

\begin{proof}[Proof of Theorem \ref{thm:schauder_WnotH}]
Suppose that \(x\in\BAN\) is the initial displacement of \(\cB\) relative to \(\B\).
Because the Schauder basis 
is 
orthogonal 
in \(\CM\), 
there exist \(\CM\)-orthogonal
\(x_1\), \(x_2\), \ldots
with
\[
 \BANn{x - (x_1+\ldots+x_n)} \quad\leq\quad 2^{-n-1} \qquad\text{ for } n\geq1\,.
\]
It suffices to take \(x_n = \sum_{k=r_{n-1}}^{r_n-1} \BANd{e_k^*}{x} e_k\in\CM\) for a sufficiently rapidly 
increasing sequence \(r_0=1 < r_1 < r_2 < \ldots\).
By the triangle inequality, for \(n\geq1\),
\begin{multline}\label{eq:norm-bound}
 \BANn{x_{n+1}} \quad\leq\quad \BANn{x - (x_1+\ldots+x_n)} + \BANn{(x_1+\ldots+x_n+x_{n+1})-x}\\
 \quad\leq\quad 2^{-n-1} + 2^{-n} \quad < \quad 2^{-n+1}\,.
\end{multline}
The \(\CM\)-orthogonal finite-dimensional projections
\(\Pi_n:z \mapsto \sum_{k=r_{n-1}}^{r_n-1} \CMd{e_k^*}{z} e_k\)
decompose 
\(\B\) and \(\cB\) 
following Theorem \ref{thm:decomposition}:
\begin{align*}
 \B(t) \quad&=\quad \Pi_1\B(t) + \Pi_2\B(t) + \ldots\,,\\
 \cB(t) \quad&=\quad \Pi_1\cB(t) + \Pi_2\cB(t) + \ldots \,.
\end{align*}
Thus \(\B_n\) is Brownian motion on the finite-dimensional subspace spanned by \(e_{r_{n-1}}, \ldots, e_{r_n-1}\).
Up to the time of coupling of \(\B_n\) and \(\cB_n\), construct \(\cB_n=R_{x_n}\B_n+x_n\) to be the reflection of \(B_n\)
started at \(x_n\) using the reflection map \(R_{x_n}\).

Let \(T_n\) be the time of coupling of \(\B_n\) and \(\cB_n\); Theorem \ref{thm:xinH} 
(really the elementary theory of finite-dimensional reflection coupling)
implies
that 
\(T_n\) 
is distributed as
the first time for scalar Brownian motion to 
hit \(0\)
when run at rate \(1\)
and started at \(\frac12\CMn{x_n}\).

Because real Brownian motion is a continuous martingale it follows that
\begin{equation}\label{eq:BM-deviation}
 \Prob{\sup\left\{\BANn{\cB_n(t)-\B_n(t)}\; :\; 0\leq t \leq T_n\right\} \geq \frac12 n (n+1) \CMn{x_n}}
=
 \frac{1}{n(n+1)}\,.
\end{equation}
By subadditivity of probability and the bound \eqref{eq:norm-bound},
\begin{multline}\label{eq:total-BM-deviation}
 \Prob{\sup\left\{\BANn{\cB_n(t)-\B_n(t)}\; :\; 0\leq t \leq T_n\right\} \geq n (n+1) 2^{-n-1} \text{ for some }n\geq N}\\
 \quad\leq\quad
 \sum_{n=N}^\infty\frac{1}{n (n+1)}\quad=\quad \frac{1}{N}\,.
\end{multline}
For times \(t > T_1\vee\ldots\vee T_{N-1}\), 
coupling means that \(\cB_1(t)=\B_1(t)\), \ldots, \(\cB_{N-1}(t)=\B_{N-1}(t)\).
Using the triangle inequality, \eqref{eq:total-BM-deviation} implies that, with probability at least \(1-1/N\),
\begin{multline}
 \sup\left\{\BANn{\cB(t)-\B(t)}\;:\;t>T_1\vee\ldots\vee T_{N-1} \right\} \quad\leq\quad \\
 \sum_{n=N}^\infty\sup\left\{\BANn{\cB_n(t)-\B_n(t)}\;:\;t>T_1\vee\ldots\vee T_{N-1}\right\}
 \quad\leq\quad 
 \sum_{n=N}^\infty n(n+1) 2^{-n-1}\,.
\end{multline}
Consequently, using the first Borel-Cantelli lemma,
\(\BANnt{\cB(t)-\B(t)} \to 0\) almost surely as \(t\to\infty\).
\end{proof}

As a corollary, the case when \(\BAN\) is a Hilbert space is now completely settled.
\begin{cor}\label{cor:Hilbert}
 If \(\BAN\) is 
 a separable Hilbert space,
 then any \(\BAN\)-valued Brownian motion corresponding to a Cameron-Martin space \(\CM\hookrightarrow\BAN\) can be coupled at time \(\infty\) from any two starting points in \(\BAN\).
\end{cor}
\begin{proof}
Recall the continuous injection \(A:\CM\hookrightarrow\BAN\).
By Prokorov's characterization of Gaussian measures on the Hilbert space \(\BAN\) (see \citealp[Theorem 2.3 and following remark]{Kuo-1975}),
\(A^*A\) is positive definite (n.b.~\(A\) is injective) and trace-class.
Its spectral decomposition therefore can be expressed in terms of finite-dimensional spaces,
and for any non-zero eigenvalue $\lambda$, the renormalized operator \(A/\lambda\) restricted to \(\text{Ker}(A^*A-\lambda^2)\) is an isometry.
Let \(v_1\), \(v_2\), {\ldots} be an orthonormal basis of \(\CM\) using eigenvectors of \(A^*A\):
by an eigenvalue argument, \(Av_1\), \(Av_2\), {\ldots} are \(\BAN\)-orthogonal.
By injectivity of \(A\) and its dense image, \(Av_1\), \(Av_2\), {\ldots} form a
\(\BAN\)-orthogonal basis, hence a Schauder basis.
\end{proof}

For more general Banach spaces, the methods of Theorem \ref{thm:schauder_WnotH} immediately supply an approach which works
at least for some initial displacements \(x\in\BAN\setminus\CM\).
\begin{cor}\label{cor:decomposition-of-initial-displacements}
 Given two initial starting points \(x\neq y\) in \(\BAN\),
 suppose it is possible to find orthogonal \(z_1\), \(z_2\) ,\ldots in \(\CM\) such that \(x-y=z_1+z_2+\ldots\), 
 with the sum converging conditionally in \(\BAN\).
 Then we can construct a coupled pair of Banach-valued Brownian motions \(\B\) and \(\cB\)
 starting from \(x\) and \(y\) which almost surely couple at time \(\infty\).
\end{cor}

\begin{rem}
\citet[Corollary 4.2, page 66]{Kuo-1975} observes that one can always construct Banach spaces \(\BAN_0\) 
with
 \(\CM{\hookrightarrow}\BAN_0\hookrightarrow\BAN\)
 and \(\BAN_0\) strictly containing \(\CM\),
 such that there is an orthonormal basis of \(\CM\) which forms a Schauder basis for \(\BAN_0\).
\end{rem}

\section{Conclusion}\label{sec:conc}
In conclusion we discuss a simple open question (Subsection \ref{sec:open}),
and indicate some further possible lines of research (Subsection \ref{sec:further-work}).

\subsection{An open question}\label{sec:open}
This paper has shown that Brownian couplings at time \(\infty\) are always possible if \(\BAN\) supports a Schauder basis
which is also \(\CM\)-orthogonal (Theorem \ref{thm:schauder_WnotH}),
and in particular that they are always possible in the important special case when \(\BAN\) is a Hilbert space
(Lemma \ref{cor:Hilbert}).
It is therefore natural to ask
\begin{qn}\label{qn:open}
Given an abstract Wiener space \(\BAN^*\hookrightarrow\CM{\hookrightarrow}\BAN\),
is it always possible to produce a coupling at time \(\infty\) for two 
\(\BAN\)-valued Brownian motions started from arbitrary starting points in \(\BAN\)?
(Or if not, then what Banach-space geometry property for \(\BAN^*\hookrightarrow\CM{\hookrightarrow}\BAN\) corresponds to the  
probabilistic property that Brownian coupling at time \(\infty\) is always possible?) 
\end{qn}

Certainly there are abstract Wiener spaces possessing the property of Brownian coupling at time \(\infty\)
which do not possess a \(\CM\)-orthogonal Schauder basis. 
Recall that \(\BAN\) has the finite dimensional decomposition property (FDD) if it supports finite-dimensional subspaces \(\BAN^{(n)}\) such that
\(x=\sum_{n=1}^\infty x_n\) (conditionally convergent) for unique \(x_n\in\BAN^{(n)}\): to be compatible with
the Cameron-Martin geometry, we must further require that the \(\BAN^{(n)}\) form an orthogonal decomposition of \(\CM\).
Note there exist separable Banach spaces \(\BAN\)
with FDD but without a Schauder basis \citep{Casazza-2001}.
The proof of Theorem \ref{thm:schauder_WnotH} shows that Brownian couplings at time \(\infty\) are always possible 
when \(\BAN\) has an FDD compatible with the Cameron-Martin geometry.

Adaptation of work of \cite{Terenzi-1994} 
permits a modest advance on the results described above.
We 
outline the argument 
briefly.

Recall that a Markushevich basis is a particular kind of biorthogonal system
(a sequence of pairs $(x_n,x_n^\ast)_{n}\in \BAN\times \BAN^\ast$ is said to be a \emph{bi-orthogonal system} whenever
\(x_n^\ast(x_m)=\delta_{n,m}\)):
\begin{defn}[Markushevich basis]
A bi-orthogonal system that is 
\begin{enumerate}
 \item \emph{fundamental}  (\(\overline{\text{span}}^{\|\cdot\|_{\BAN}}\{x_n\}_{n\in \N}=\BAN\))
 \item and \emph{total}    (\(\overline{\text{span}}^{w^\ast}\{x_n^\ast\}_{n\in \N}=\BAN^\ast\))
\end{enumerate}
is 
a \emph{Markushevich basis} (M-basis).
(Here $\overline{\text{span}}^{w^\ast}$ 
is
closure of 
linear span under 
weak$^\ast$ topology.)
\end{defn}

The modest advance is as follows:
Brownian coupling at time \(\infty\) 
is possible
if
\(\BAN\) supports a \emph{norming} M-basis $(x_n,x_n^\ast)_{n}\in \BAN\times \BAN^\ast$, \(n=1,2,\ldots\),
with
\(\CM\)-orthogonal \(x_n\): ``norming'' means that
\[
\vertiii{x}
\quad=\quad 
\sup \{\langle x,x^\ast\rangle \ : \ x^\ast \in 
(\overline{\operatorname{span}}^{\BANn{\cdot}}\{x_n^\ast\}_n)\cap B_{(\BAN^\ast, \|\cdot\|_{\BAN^\ast} )}\}
\]
defines a norm on \(\BAN\) such that \(\lambda \BANn{x}\leq \vertiii{x}\leq \BANn{x}\) for some $0<\lambda\leq 1$.
Here $ B_{(\BAN^\ast, \|\cdot\|_{\BAN^\ast} )}$  denotes the unit ball in the Banach space $\BAN^\ast$.
Note that the argument used in the proof of Theorem \ref{thm:schauder_WnotH} 
will apply if any \(x\in\BAN\) can be represented 
as the limit of partial sums of orthogonal elements of \(\CM\), so this is the objective.
\cite{Terenzi-1994} establishes this representation in terms of existence of an M-basis for \emph{any} separable Banach space,
but without imposing any requirement of \(\CM\)-orthogonality.
We now describe the steps required to adapt \cite{Terenzi-1994} to show that the argument
can be sufficiently
related to the underlying Cameron-Martin geometry to capture enough \(\CM\)-orthogonality
to obtain the required representation.

\emph{Step 1:}
Assume 
existence of
a \emph{norming} M-basis whose elements are orthogonal in \(\CM\).

\emph{Step 2:} Further adjust the norming M-basis to be \emph{bounded}
(so that
there is a constant $0<M<\infty$ such that $\sup_n \{\BANn{x_n}\cdot \|x_n^\ast\|_{\BAN^\ast}\}\leq M$).
This uses 
the approach of
\cite{OvsepianPelczynski-1975}: 
the M-basis is adjusted using Haar unitary matrix transformations on disjoint finite-dimensional subspaces,
hence retaining \(\CM\)-orthogonality
while preserving the disjoint subspace decomposition.

\emph{Step 3:} Now note \citet[Lemma 1.40]{HajekSantaluciaVanderwerffZizler-2008}:
\begin{lem}
If $\{x_n,x_n^*\}_n$ 
is a 
bounded norming M-basis
for $\BAN$
then 
there 
are
integers $r(1)<r(2)<...$ such that for any $x \in \BAN$ and every $m \in \N$ 
there is an element $v_m$ in span$\{x_n\}_{n=r(m)+1}^{r(m+1)} $ 
with
\[
x\quad=\quad
\lim_{m\to \infty} \Big[\sum_{n=1}^{r(m)} \langle x,x_n^*\rangle x_n + v_m\Big]\,.
\]
\end{lem}
\begin{rem}
Note that by a result of \citet{Fonf-1986}, the above property holds if and only if the M-basis is norming; cf.\ also \citet{Terenzi-1994}.
\end{rem}
\noindent
If (selecting a subsequence depending on \(x\) if necessary) we could contrive that \(v_m\to0\) in \(\BAN\) 
for the M-basis obtained in \emph{Steps 1} and \emph{2},
then we would have obtained the required expression of \(x\) as the limit of partial sums of orthogonal
elements of \(\CM\).

\emph{Step 4:} We seek a block perturbation providing the subsequence required in \emph{Step 3}.
A \emph{block perturbation} $\{y_n,y_n^\ast\}_n $ of $\{x_n,x_n^\ast\}_n $ amounts (for our purposes) to finding an increasing sequence $\{q_m\}_{m\in\N}$ of positive integers such that for every $m$
\[
\begin{split}
\text{span}\{y_n\}_{n=q_m+1}^{q_{m+1}}=\text{span}\{x_n\}_{n=q_m+1}^{q_{m+1}}\quad \text{ and }\quad  \text{span}\{y_n^\ast\}_{n=q_m+1}^{q_{m+1}}=\text{span}\{x_n^\ast\}_{n=q_m+1}^{q_{m+1}}.
\end{split}
\]
If it is not possible to find a 
block perturbation
leading to $v_m\to 0$ in $\BAN$ as described in \emph{Step 3}, then the \cite{Terenzi-1994} argument uses further careful constructions of block
perturbations of the M-basis, 
together with the \citet{Dvoretzky-1961} Hilbert space approximation Theorem,
to produce a new norming M-basis (depending on \(x\)) which 
can be used to generate successive partial sums which approximate \(x\) in \(\BAN\)-norm.
Because block perturbations are used, these successive partial sums can be expressed in terms of a sequence of orthogonal
elements of \(\CM\) (depending on $x$), and so the objective is attained.

We end this section with a couple of remarks about the last proof.
\begin{rem}
We have assumed that we can pick an M-basis for $\BAN$ whose elements are orthogonal in \(\CM\).
At 
present
we are not aware of 
useful
sufficient conditions
to
guarantee this.
Indeed we do not know if it is possible \emph{always} to find a $\CM$-orthogonal M-basis for $\BAN$,
though this is implausible.
\end{rem}
\begin{rem}
Note that 
the orthogonal M-basis 
needs to be
\emph{norming}.
The norming property is guaranteed when $\BAN$ is quasi-reflexive (i.e., the canonical image of $\BAN$ in its second dual has finite co-dimension), 
following
a result of \citet{Petunin-1964} (cf.\ also \citealp{Singer-1961-62}).
However \citep{DavisLindenstrauss-1972} 
if $\BAN$ is not quasi-reflexive
then it is always possible to find a total subspace of $\BAN^\ast$ that is not norming in $\BAN$.
Hence the norming assumption is necessary when dealing with general separable Banach spaces.
\end{rem}

\subsection{Further work}\label{sec:further-work}
Apart from addressing the open question \ref{qn:open},
one 
might
also ask whether 
one could
additionally 
couple functionals of the two Banach-valued Brownian motions,
such as for example their L\'evy stochastic areas. 
This can be done in the finite-dimensional
case \cite{BenArousCranstonKendall-1995,Kendall-2007,Kendall-2009d}.
However L\'evy stochastic areas are quadratic objects, so it seems likely that only rather limited results can be obtained in infinite-dimensional cases.

\cite{BanerjeeKendall-2014} have shown that geometric criteria tightly constrain Markovian maximal couplings 
for (finite-dimensional) smooth elliptic diffusions:
in dimension \(2\) and above, the existence of a Markovian maximal coupling forces the diffusion
to be Brownian motion on a simply-connected space of constant curvature, with drift given by a combination
of a Killing vectorfield and (in the Euclidean case) a vectorfield related to scaling symmetry.
It is evident from Theorem \ref{thm:xinH} that the notion of maximality does not extend usefully to the Banach-space case;
but is there a weaker form of optimality which also imposes special requirements in some underlying geometry?

Turning to potential applications, it is natural to speculate about applications
of these coupling constructions to multi-scale problems. 
For example multiresolution image analysis models an image as an infinite hierarchy of features of progressively
finer resolution:
there are interesting phase transition phenomena linked to image analysis issues
\citep{KendallWilson-2003},
while a Coupling-from-the-Past algorithm 
has been developed for a point process example \citep{AmblerSilverman-2010}.
The ``coupling at time \(\infty\)'' constructions of Theorem \ref{thm:schauder_WnotH}
are suggestive for such problems. 
Finally, Hairer and others (for example, \citealp{Hairer-2002,HairerMattingley-Scheutzow-2011})
have discussed applications of rather specific couplings to certain SPDE. 
It would be interesting to relate the abstract considerations of the present paper to this applied context.

\nocite{}
\bibliographystyle{chicago}
\bibliography{Banach}

\newpage


\end{document}